\pdfoutput=1
\documentclass[a4paper,11pt]{article}
\usepackage{hyperref}
\hypersetup{
    colorlinks=true,
    linktoc=all,
    linkcolor=red,     
    citecolor=blue,     
}
\usepackage{graphicx}
\usepackage{amsmath,amssymb,amsfonts,amsthm,graphics,enumitem}
\usepackage{latexsym, bm}
\usepackage{multicol}
\usepackage{comment}
\usepackage{indentfirst}

\usepackage{setspace}
\newtheorem{theorem}{Theorem}[section]
\newtheorem{lemma}{Lemma}[section]
\newtheorem{claim}{Claim}

\newtheorem{corollary}{Corollary}[section]
\newtheorem{proposition}{Proposition}[section]

\newtheorem{conjecture}{Conjecture}[section]

\newcommand{\ignore}[1]{}

\begin{document}
\begin{spacing}{1}
\date{}
\title{Connectivity keeping pendant extensions of paths in $k$-connected triangle-free graphs\thanks{The work is supported by NSFC(11871015) and CGLSTDFFJ(2023L3003). E-mail: menghanma664@163.com (M. Ma), qliu@fzu.edu.cn (Q. Liu), 1519036064@qq.com (L. Zhang), yhong@fzu.edu.cn (Y. Hong).}}

\author{Menghan Ma$^a$, Qinghai Liu$^{a,c}$\footnote{Corresponding author.}, Liping Zhang$^{a}$, Yanmei Hong$^{b,c}$\\[1mm]
{\small \it
$^a$Center for Discrete Mathematics, Fuzhou University, Fuzhou, Fujian, 350108 China}\\
{\small \it
$^b$School of Mathematics and Statistics, Fuzhou University, Fuzhou, Fujian, 350108 China}\\
{\small \it
$^c$Key Laboratory for Operations Research and Cybernetics of Fujian Universities}
}
\maketitle
\begin{abstract}
Motivated by Mader's conjecture on connectivity keeping trees, we study the corresponding problem in triangle-free graphs.

For positive integers $k,m$, we first prove that every $k$-connected triangle-free graph $G$ with $\delta(G)\geq k+\left \lceil \frac{m}{2} \right \rceil $ contains a path $P$ of order $m$ such that $G-V(P)$ remains $k$-connected. Moreover, if $\kappa(G)=k$, then for every end $F$ of $G$
and every vertex $v\in F$, the path $P$ can be chosen in $G[F]$ with $v$ as an endvertex. If $\kappa(G)\geq k+1$, then any vertex of $G$ can be chosen as an endvertex of $P$. This result confirms the conjecture of Luo, Tian and Wu on connectivity keeping trees in bipartite graphs for all paths, and it also sharpens Fujita's path theorem to the optimal minimum degree bound.

For an integer $i$ with $1\leq i\leq m-1$, let $P_m^+(i)$ denote the tree obtained from a path of order $m-1$ by adding one pendant vertex adjacent
to its $i$th vertex. Let $(X,Y)$ be the bipartition of $P_m^+(i)$. We further prove that every $k$-connected triangle-free graph $G$ with $\delta(G)\geq k+\max\{|X|,|Y|\}$ contains a subgraph $T\cong P_m^+(i)$ such that $G-V(T)$ remains $k$-connected. The minimum degree bound is sharp, and the result confirms the bipartite conjecture for all pendant extensions of paths.
\medskip


\end{abstract}

\section{Introduction}

In this paper, all graphs are considered to be finite, undirected, and simple. We refer to the book \cite{Bondy} for any undefined notation and terminology in the following.

For a graph $G$, we denote its {\it vertex set}, {\it edge set}, {\it connectivity}, and {\it minimum degree} by $V(G)$, $E(G)$, $\kappa(G)$, and $\delta(G)$, respectively. For a vertex set $U\subseteq V(G)$, $G[U]$ denotes the subgraph induced by $U$ and $G-U=G[V(G)\setminus U]$.
The {\it connectivity} $\kappa(G)$ of $G$ is the cardinality of the minimum vertex subset $S\subseteq V(G)$ such that $G-S$ is disconnected or trivial. A graph $G$ is {\it $k$-connected} if $\kappa(G)\geq k$.

In 1972, Chartrand, Kaugars and Lick \cite{Chartrand} first proved that every $k$-connected graph $G$ with  $\delta(G)\ge\left\lfloor\frac{3k}{2}\right\rfloor$ contains a vertex $u$ such that $\kappa(G-u)\geq k$. In 2008, Fujita and Kawarabayashi \cite{Fujita} proved that every $k$-connected graph $G$ with $\delta(G)\ge\left\lfloor\frac{3k}{2}\right\rfloor+2$ contains an edge $uv$ such that $\kappa(G-\{u,v\})\geq k$, and they proposed a conjecture about a redundant subgraph in a $k$-connected graph.

\begin{conjecture}[\upshape Fujita and Kawarabayashi~\cite{Fujita}]\label{Fujita} For all positive integers $k,m$, there is a (least) non-negative integer $f_k(m)$ such that every $k$-connected graph $G$ with $\delta(G)\ge\left\lfloor\frac{3k}{2}\right\rfloor+f_k(m)-1$ contains a connected subgraph $W$ of order $m$ such that $\kappa(G-V(W))\geq k$.
\end{conjecture}

In 2010, Mader \cite{Mader1} proved the validity of Conjecture \ref{Fujita}, explicitly established that $f_k(m)=m$, and showed that the subgraph $W$ can be chosen as a path.

\begin{theorem}[\upshape Mader~\cite{Mader1}]\label{mader1}
For all positive integers $k,m$, every $k$-connected graph $G$ with $\delta(G)\ge\left\lfloor\frac{3k}{2}\right\rfloor+m-1$ contains a path $P$ of order $m$ such that $\kappa(G-V(P))\geq k$.
\end{theorem}

Mader further conjectured that the path in the theorem above can be replaced by any tree of the same order.

\begin{conjecture}[\upshape Mader~\cite{Mader1}]\label{Mad}
For any tree $T$ of order $m$ and any positive integer $k$, every $k$-connected graph $G$ with $\delta(G)\ge\left\lfloor\frac{3k}{2}\right\rfloor+m-1$ contains a subtree $T'\cong T$ such that $\kappa(G-V(T'))\geq k$.
\end{conjecture}

As early as 2009, the result published by Diwan and Tholiya \cite{Diwan} established the validity of Conjecture \ref{Mad} for the case when $k=1$. For the case when $k=2$, there have been some studies on special trees \cite{Hasunuma2, HongLLY, Lu, Tian1, Tian2}. In \cite{HongLiu}, Hong and Liu proved that the conjecture holds for any tree when $k\leq 3$. For general $k$, Mader \cite{Mader2} proved the conjecture holds if $\delta(G)\geq 2(k+m-1)^2+m-1$. Liu, Ying and Hong \cite{LYH} improved this to the linear bound and proved that the conjecture holds if $\delta(G)\geq 3k+4m-6$. For $k\geq 4$, Conjecture \ref{Mad} remains open. For more studies on Mader's conjecture, one can refer to the survey \cite{Tian3}.

For bipartite graphs, Luo, Tian and Wu \cite{Luo} showed that the minimum degree condition in Theorem~\ref{mader1} can be replaced by $\delta(G)\geq k+m$, and proposed the following bipartite version of Mader's conjecture.

\begin{conjecture}[\upshape Luo, Tian and Wu \cite{Luo}]\label{cluo}
For any tree $T$ with bipartition $X$ and $Y$, every $k$-connected bipartite graph $G$ with $\delta(G)\geq k+\max\{|X|,|Y|\}$ contains a tree $T'\cong T$ such that $\kappa(G-V(T'))\geq k$.
\end{conjecture}

Hong, Wu and Liu \cite{HongWuLiu} proved the conjecture for all trees when $k\leq 3$. For general $k$, Yang and Tian \cite{YangTian} showed that every $k$-connected bipartite graph $G$ with $\delta(G)\geq k+\frac{m+1}{2}$ contains a path $P$ of order $m$ such that $\kappa(G-V(P))\geq k$. This in particular settles Conjecture \ref{cluo} for all paths of odd order, a result also obtained independently by Ji \cite{Ji}. Yet for even $m$, this minimum degree bound is one larger than the conjectured bound. For non-path trees, Ji \cite{Ji2} studied connectivity keeping star-paths in $k$-connected bipartite graphs.

As bipartite graphs contain no triangles, one naturally asks the same question for triangle-free graphs. Chu, Fujita, Park and Ryu \cite{CFPR} proved that for any tree $T$ of order $m$, every $k$-connected triangle-free graph $G$ with $\delta(G)\geq 2k+3m-4$ contains a subgraph $T'\cong T$ such that $\kappa(G-V(T'))\geq k$. For paths,  Fujita \cite{Ftriangle} established the following two results.

\begin{theorem}[\upshape Fujita~\cite{Ftriangle}]\label{F1}
Let $k,m$ be positive integers. If $G$ is a $k$-connected triangle-free graph with $\delta(G)\geq k+\frac{m+1}{2}$, then $G$ contains a path $P$ on
$m$ vertices such that $G-V(P)$ is $k$-connected. Moreover, $P$ has the following additional properties according to the value of $\kappa(G)$.
\begin{enumerate}
\item If $\kappa(G)=k$, then for every end $F$ of $G$ and every vertex $v\in F$, there exists such a path $P$ with $V(P)\subseteq G[F]$ and with $v$ as an endvertex.
\item If $\kappa(G)\geq k+1$, then we can choose any vertex of $G$ as an endvertex of $P$.
\end{enumerate}
\end{theorem}
\begin{theorem}[\upshape Fujita~\cite{Ftriangle}]\label{F2}
For an integer $k\geq 3$, if $G$ is a $k$-connected triangle-free graph with $\delta(G)\geq k+1$, then $G$ contains an edge $e=xy$ such that $\kappa(G-\{x,y\})\geq k$.
\end{theorem}
In this paper, we improve the minimum degree bound of Theorem \ref{F1} to $k+\left \lceil \frac{m}{2} \right \rceil$, while retaining both additional conclusions. For even $m$, the required minimum degree is decreased by one. Consequently, our result confirms Conjecture \ref{cluo} for all paths.
\begin{theorem}\label{tripath}
Let $k,m$ be positive integers. If $G$ is a triangle-free graph with $\kappa(G)\geq k$ and $\delta(G)\geq k+\left \lceil \frac{m}{2} \right \rceil$, then $G$ contains a path $P$ on $m$ vertices such that $G-V(P)$ is $k$-connected. Moreover, $P$ has the following additional properties according to the value of $\kappa(G)$.
\begin{enumerate}
\item If $\kappa(G)=k$, then for every end $F$ of $G$ and every vertex $v\in F$, there exists such a path $P$ with $V(P)\subseteq G[F]$ and with $v$ as an endvertex.
\item If $\kappa(G)\geq k+1$, then we can choose any vertex of $G$ as an endvertex of $P$.
\end{enumerate}
\end{theorem}
We also consider trees obtained from paths by adding a single pendant vertex. For integers $m$ and $i$ with $1\leq i\leq m-1$, let $P_m^+(i)$ denote the tree obtained from a path $v_1v_2\cdots v_{m-1}$ by adding a new vertex $u$ adjacent to $v_i$. In particular, $P_m^+(1)$ and $P_m^+(m-1)$ are both paths of order $m$, whereas $P_m^+(i)$ is not a path when $2\leq i\leq m-2$. The following theorem confirms Conjecture \ref{cluo} for all such trees and establishes the corresponding result for triangle-free graphs.
\begin{theorem}\label{triangle}
Let $k,m,i$ be positive integers with $1\leq i\leq m-1$, and let $(X,Y)$ be the bipartition of $P_m^+(i)$. Every $k$-connected triangle-free graph $G$ with $\delta(G)\ge k+\max\{|X|,|Y|\}$ contains a subgraph $T\cong P_m^+(i)$ such that $\kappa(G-V(T))\geq k$. Moreover, if $\kappa(G)=k$, then $T$ can be chosen such that $T\subseteq G[F]$, where $F$ is an arbitrary end of $G$.
\end{theorem}

The minimum degree bound in Theorem \ref{triangle} is sharp. Let $a=\min\{|X|,|Y|\}$ and let $b=\max\{|X|,|Y|\}$, and consider the graph $G=K_{k+b-1,k+b-1}$. Then $G$ is a $k$-connected triangle-free graph and $\delta(G)=k+b-1$. For every subgraph $T\cong P_m^+(i)$ of $G$, the graph $G-V(T)\cong K_{k+b-a-1,k-1}$ is not $k$-connected. Thus, the bound on $\delta(G)$ in Theorem \ref{triangle} is best possible.

The rest of this paper is organized as follows. In Section \ref{preliminaries}, we summarize the notation and terminology used throughout the paper. In Section \ref{trithm}, we give the proof of Theorem \ref{tripath}. Section \ref{trifulll} is devoted to the proof of Theorem \ref{triangle}.
\section{Preliminaries}\label{preliminaries}
In this section, we introduce some notation and definitions and collect several lemmas that will be used in the proofs of our main results.

For a graph $G$ and its vertex $x$, let $d_G(x)$ and $N_G(x)$ denote
the {\it degree} and {\it neighborhood} of $x$ in $G$, respectively.
When considering a vertex subset $X\subseteq V(G)$, we define $N_G(X)=\cup_{x\in X}N_G(x)\setminus X$. For a positive integer $a$, let $[a]=\{1,2,\cdots,a\}$.

A graph $G$ is called a {\it complete graph} if every pair of distinct vertices is adjacent. The complete graph with $n$ vertices is denoted by $K_n$. A bipartite graph $G$ with bipartition $(X,Y)$ is called a {\it complete bipartite graph} if every vertex of $X$ is adjacent to every vertex of $Y$. The complete bipartite graph with $|X|=m$ and $|Y|=n$ is denoted by $K_{m,n}$.
For a subgraph $H$, write $|H|=|V(H)|$ and $N_G(H)=N_G(V(H))$. Whenever a subgraph occurs in a set intersection or union, it is identified with its vertex set.

In a graph $G$, for two disjoint paths $P=v_1v_2\dots v_k$ and $Q=u_1u_2\dots u_m$ such that $v_ku_1\in E(G)$,  we use $PQ$ to denote the path $v_1v_2\dots v_ku_1u_2\dots u_m$. We simply write $vQ$ rather than  $PQ$ if $P$ only contains one vertex $v$, and write $Pu$ rather than $PQ$ if $Q$ only contains one vertex $u$. For a path $P$ and any vertices $x,y\in V(P)$, we use $P[x,y]$ to denote the subpath of $P$ from $x$ to $y$.

A {\it vertex cut set} of a graph $G$ is a subset $S\subseteq V(G)$ such that $G-S$ is disconnected or trivial. If $|S|=k$, then we say $S$ is a {\it $k$-cutset} of $G$. Given a vertex cut set $S$, a nonempty union $F$ of components of $G-S$ with $G-(S\cup F)\neq \emptyset$ is called a {\it semifragment} of $G$ to $S$, its complement $\bar{F}=G-(S\cup F)$ is the {\it complementary semifragment}. If $S$ is a minimum cut set of $G$, then $F$ is called a {\it fragment} of $G$ to $S$ if $N_G(F)=S$. A fragment that contains no other fragments of $G$ is called an {\it end} of $G$. Every non-complete graph contains at least two ends. For a set $S$, $K(S)$ denotes the complete graph on vertex set $S$. The {\it completion} of $S\subseteq V(G)$ in $G$, denoted by $G\left \langle S \right \rangle $, is the graph $G\cup K(S)$.

\begin{lemma}[\cite{Mader1}]\label{mader[S]}
Let $G$ be a graph with $\kappa(G)=k$ and let $F$ be a fragment of $G$ to $S$. Then $\kappa(G\left \langle S \right \rangle-\bar{F})\geq k$. In particular, if $F$ is an end of $G$ with $|F|\geq 2$, then $\kappa(G\left \langle S \right \rangle-\bar{F})\geq k+1$.
\end{lemma}
\begin{lemma}\label{subend}
Let $k$ be a positive integer, let $G$ be a graph with $\kappa(G)=k$, let $S$ be a $k$-cutset of $G$ and $F$ be an end of $G$ to $S$. If $\emptyset\neq W\subsetneq F$ and $\kappa(G-W)\geq k$, then $\kappa(G-W)=k$ and there is an end $F'\subseteq F$ of $G-W$ such that $N_G(W)\cap V(F')\neq \emptyset$.
\end{lemma}
\begin{proof}
Assume $F$ is an end of $G$ and $W\subsetneq F$ such that $\kappa(G-W)\geq k$. Since both $F\setminus W$ and $\bar F$ are nonempty, $S$ is still a $k$-cutset of $G-W$. Thus, $\kappa(G-W)=k$ and $F\setminus W$ is a fragment of $G-W$ to $S$. Let $H=G-W$. Choose an end $F'$ of $H$ contained in the fragment $F\setminus W$. If $N_G(W)\cap V(F')=\emptyset$, then $N_G(F')=N_H(F')$ has size $k$ and thus $F'$ is a fragment of $G$ properly contained in $F$, a contradiction.
\end{proof}
\begin{lemma}\label{common}
Let $k$ be a positive integer, let $G$ be a graph, let $S$ be a $k$-cutset of $G$ and $A$ be a semifragment of $G$ to $S$ with $|A|,|\bar{A}|\geq k$, where $\bar{A}=G-(S\cup A)$. If $G\left \langle S \right \rangle-A$ and $G\left \langle S \right \rangle-\bar{A}$ are both $k$-connected, then for every $S'\subseteq V(G)$ with $|S'|<k$, there exists a component $C$ of $G-S'$ such that $(A\cup \bar{A})\setminus S'\subseteq V(C)$.
\end{lemma}
\begin{proof}
Fix $S'\subseteq V(G)$ with $|S'|<k$ and suppose that the condition fails. Since $|A|,|\bar{A}|\geq k$, we obtain that $A\setminus S',\bar{A}\setminus S'\neq \emptyset$ and there exists a semifragment $B$ of $G-S'$ such that $A\cap B\neq \emptyset$ and $\bar{A}\cap \bar{B}\neq \emptyset$, where $\bar{B}=G-(S'\cup B)$.

Let $X\subseteq V(G)$ be a set with $|X|\leq k-1$. Since $G\left \langle S \right \rangle-\bar{A}$, $G\left \langle S \right \rangle-A$ are both $k$-connected and $S\setminus X\neq \emptyset$, $G\left \langle S \right \rangle-X$ is connected. Thus, $G\left \langle S \right \rangle$ is $k$-connected. As $\bar{B}\neq \emptyset$, then $A\cap B\neq \emptyset$ implies that $(S\cap B)\cup(S\cap S')\cup(S'\cap A)$ is a cut set of $G\left \langle S \right\rangle$ and thus $|(S\cap B)\cup(S\cap S')\cup(S'\cap A)|\geq k$. Similarly, as $A\neq \emptyset$, then $\bar{A}\cap \bar{B}\neq \emptyset$ implies that $(S\cap S')\cup(S\cap \bar{B})\cup(S'\cap \bar{A})$ is a cut set of $G\left \langle S \right \rangle$ and thus $|(S\cap S')\cup(S\cap \bar{B})\cup(S'\cap \bar{A})|\geq k$. Therefore,
\begin{align*}
2k\leq  |(S\cap B)\cup(S\cap S')\cup(S'\cap A)|+|(S\cap S')\cup(S\cap \bar{B})\cup(S'\cap \bar{A})|=|S|+|S'|<2k,
\end{align*}
a contradiction. This completes the proof of Lemma \ref{common}.
\end{proof}
\begin{lemma}\label{lem:add-vertex-outside-end}
Let $k$ be a positive integer, let $G$ be a graph with $\kappa(G)=k$, let $S$ be a $k$-cutset of $G$ and let $F$ be an end of $G$ to $S$. Let $Q$ be a set of new vertices such that every vertex of $Q$ has at least $k$ neighbors in $G$ and no neighbor in $F$, where the edges in $Q$ are arbitrary. Denote the resulting graph by $G'$. Then $\kappa(G')=k$, and $F$ is still an end of $G'$ to $S$.
\end{lemma}
\begin{proof}
Let $X\subseteq V(G')$ with $|X|\le k-1$, and let $X_G=X\cap V(G)$. Since $G$ is $k$-connected, $G-X_G$ is connected. For every $u\in Q\setminus X$, the vertex $u$ has at least $k$ neighbors in $V(G)$. Since $|X_G|\le k-1$, at least one of these neighbors lies in $V(G)\setminus X_G$. Thus every vertex of $Q\setminus X$ is adjacent to the connected graph $G-X_G$. Consequently, $G'-X$ is connected. Therefore, $G'$ is $k$-connected.

Since no vertex of $Q$ has a neighbor in $F$, we have $N_{G'}(F)=N_G(F)=S$. So $S$ remains a $k$-cutset of $G'$, and thus $\kappa(G')=k$. If $A\subsetneq F$ is a fragment of $G'$, then $N_G(A)=N_{G'}(A)$ since no new vertex has a neighbor in $F$. Thus $A$ is also a fragment of $G$, contradicting the assumption that $F$ is an end of $G$. Therefore, $F$ remains an end of $G'$ to $S$.
\end{proof}
\begin{lemma}\label{Fk+m}
Let $k,m$ be two positive integers, let $G$ be a triangle-free graph with $\kappa(G)=k$ and $\delta(G)\geq k+\left \lceil \frac{m}{2} \right \rceil$. If $F$ is a fragment of $G$, then $|F|\geq k+m$.
\end{lemma}
\begin{proof}
We first show that $E(G[F])\neq \emptyset$. If not, then for every $u\in F$, $d_G(u)\leq |N_G(F)|=k<k+\left \lceil \frac{m}{2} \right \rceil$, contradicting $\delta(G)\geq k+\left \lceil \frac{m}{2} \right \rceil$. Take $uv\in E(G[F])$. Since $G$ is triangle-free, $N_G(u)\cap N_G(v)=\emptyset$. Moreover, since $N_G(u)\cup N_G(v)\subseteq F\cup N_G(F)$, $|F|\geq d_G(u)+d_G(v)-|N_G(F)|\geq 2(k+\left \lceil \frac{m}{2} \right \rceil)-k\geq k+m$.
\end{proof}
For a tree $T$ with bipartition $(X,Y)$, define $\max(T)=\max\{|X|,|Y|\}$.
\begin{lemma}\label{vT^*}
Let $G$ be a triangle-free graph and let $T\cong P_j^+(i)$ be a subgraph of $G$. Then $|N_G(v)\cap V(T)|\leq \max(T)$ for any $v\in V(G)$.
\end{lemma}
\begin{proof}
Choose an arbitrary vertex $v\in V(G)$. If $T$ is a path, then the assertion follows from the fact that $G$ is triangle-free. Thus, assume that $T$ is not a path. Let $t$ be a leaf of $T$ such that $T-t$ is a path, and let $N_T(t)=\{t'\}$.

If $t\notin N_G(v)$, since $T-\{t\}$ is a path, $|N_G(v)\cap V(T)|\leq \max(T-\{t\})\leq \max(T)$. If $t\in N_G(v)$, then $t'\notin N_G(v)$ since $G$ is triangle-free. Since $T-\{t,t'\}$ consists of two paths $P_1,P_2$, $|N_G(v)\cap V(T)|\leq \max(P_1)+\max(P_2)+1=\max(T)$.
This completes the proof.
\end{proof}

\section{Proof of Theorem \ref{tripath}}\label{trithm}
\begin{lemma}\label{com}
Let $k,m$ be positive integers, let $G$ be a triangle-free graph with $\kappa(G)=k$ and $\delta(G)\geq k+\left \lceil \frac{m}{2} \right \rceil$, and let $F$ be a fragment of $G$ to $S$. If $P\subseteq G[F]$ is a path of order at most $m$ and $\kappa(G\left \langle S \right \rangle-(\bar{F}\cup V(P)))\geq k$, then $\kappa(G-V(P))\geq k$.
\end{lemma}
\begin{proof}
Suppose, to the contrary, that $\kappa(G-V(P))<k$. Let $H=G-V(P)$ and let $S'$ be a cut set of $H$ with $|S'|\leq k-1$. Since $G$ is triangle-free, each vertex of $H$ has at most $\left \lceil \frac{|P|}{2} \right \rceil\leq \left \lceil \frac{m}{2} \right \rceil$ neighbors on $P$. Thus, $\delta(H)\geq \delta(G)-\left \lceil \frac{m}{2} \right \rceil\geq k$.

Now we verify the conditions of Lemma \ref{common} for $H$ and $S$. Let $F'=F\setminus V(P)$. By Lemma \ref{Fk+m}, $|F|,|\bar{F}|\geq k+m$. Since $|P|\leq m$, $|F'|\geq k$. Moreover, $V(H)=F'\cup S\cup\bar F$, and there are no edges between $F'$ and $\bar F$. Thus, $F'$ and $\bar F$ are semifragments of $H$ to $S$. By Lemma \ref{mader[S]}, $\kappa(G\left \langle S \right \rangle-F)\geq k$ and thus $\kappa(H\left \langle S \right \rangle-F')\geq k$. We also obtain $\kappa(H\left \langle S \right \rangle-\bar{F})\geq k$ from $\kappa(G\left \langle S \right \rangle-(\bar{F}\cup V(P)))\geq k$. Therefore, Lemma \ref{common} implies that $H-S'$ contains a component $C$ such that $(F'\cup\bar{F})\setminus S'\subseteq V(C)$.

Let $C'\neq C$ be another component of $H-S'$. Then $V(C')\subseteq S$ and $N_H(V(C'))\subseteq S\cup S'$. It follows from $\delta(H[V(C')])\geq \delta(H)-|S'|\geq k-(k-1)\geq 1$ that $E(C')\neq \emptyset$. Choose $xy\in E(C')$. As $H$ is triangle-free, $N_H(x)\cap N_H(y)=\emptyset$. Thus, $2k\leq d_H(x)+d_H(y)\leq|S|+|S'|\leq 2k-1$, a contradiction. This completes the proof of Lemma \ref{com}.
\end{proof}
\begin{lemma}\label{m=2}
Let $k$ be a positive integer, let $G$ be a triangle-free graph. The following statements hold.
\begin{enumerate}[label=(\arabic*)]
\item If $\kappa(G)=k$ and $F$ is an end of $G$ to $S_F$ with $|F|\geq 2$, then for every vertex $u\in F$ there exists $w\in N_G(u)\cap F$ such that $\kappa(G\left \langle S_F \right \rangle-(\bar{F}\cup \{u,w\}))\geq k$.
\item If $\kappa(G)\geq k+1$, then for every vertex $u\in V(G)$ there exists $w\in N_G(u)$ such that $\kappa(G-\{u,w\})\geq k$.
\end{enumerate}
\end{lemma}
\begin{proof}
We use one argument for both assertions. For $(1)$, let $H=G\left \langle S_F \right \rangle-\bar{F}$ and let $S=S_F$. For $(2)$, let $H=G$, $F=V(G)$ and $S=\emptyset$. In either case, $V(H)=F\cup S$ and $\kappa(H)\geq k+1$. For $(1)$, the connectivity bound follows from Lemma \ref{mader[S]}. Thus, $\delta(H)\geq k+1$. Moreover, every edge $e\in E(H)$ with $V(e)\cap F\neq \emptyset$ is an edge of $G$. Since $G$ is triangle-free, $H[R]=G[R]$ is also triangle-free for any $R\subseteq F$.

Suppose, to the contrary, that there exists $u\in F$ such that $\kappa(H-\{u,w\})<k$ for every $w\in W$, where $W=N_{G[F]}(u)$. Since $H[F]$ is triangle-free, $W$ is an independent set. Let $H'=H-\{u\}$. The above assumption and $\kappa(H)\geq k+1$ imply that $\kappa(H)=k+1$, $\kappa(H')=k$ and there exists a $k$-cutset of $H'$ containing $w$ for every $w\in W$.


We say a pair $(S^*,F^*)$ is {\it good} if $S^*$ is a $k$-cutset of $H'$ with $S^*\cap W\neq \emptyset$ and $F^*$ is a fragment of $H'$ to $S^*$ with $F^*\cap S=\emptyset$. Since $H'[S]$ is complete, such a good pair exists.
\begin{claim}\label{Fsize}
If $(S^*,F^*)$ is good, then $F^*\cap W\neq \emptyset$ and $|F^*|\geq k+1$.
\end{claim}
\begin{proof}
Since $\kappa(H)\geq k+1$, $F^*\cap W\neq \emptyset$ (i.e., $F^*\cap N_G(u)\neq \emptyset$). Otherwise $S^*$ is a $k$-cutset of $H$, a contradiction. Choose $w\in F^*\cap W$. Since $S^*\cap W\neq \emptyset$ and $W$ is an independent set in $H'$, $|N_{H'}(w)\cap S^*|\leq k-1$. Note that $d_{H'}(w)= d_H(w)-1\geq k$. Then $|N_{H'}(w)\cap F^*|= d_{H'}(w)-|N_{H'}(w)\cap S^*|\geq 1$. Choose $w'\in N_{H'}(w)\cap F^*$. Since $u,w,w'\in F$ and $H[F]$ is triangle-free, $w'u\notin E(H)$ and thus $d_{H'}(w')=d_{H}(w')\geq k+1$. Moreover, since $N_{H'}(w)\cap N_{H'}(w')=\emptyset$ and $N_{H'}(w)\cup N_{H'}(w')\subseteq F^*\cup S^*$, $|F^*|\geq d_{H'}(w)+d_{H'}(w')-|S^*|\geq 2k+1-k\geq k+1$.
\end{proof}
Choose a good pair $(S_1,F_1)$ with $|F_1|$ as small as possible. Let $\bar{F_1}=H'-(S_1\cup F_1)$. By Claim \ref{Fsize}, $F_1\cap W\neq \emptyset$. Choose $w\in N_H(u)\cap F_1$. Let $S_2$ be a $k$-cutset of $H'$ containing $w$, let $F_2$ be a fragment of $H'$ to $S_2$ and $\bar{F_2}=H'-(S_2\cup F_2)$. We use the following notation for the nine intersections:
\begin{align*}
A_1=F_1\cap F_2, A_2=S_1\cap F_2,A_3=\bar F_1\cap F_2,\\
A_4=F_1\cap S_2,A_5=S_1\cap S_2,A_6=\bar F_1\cap S_2,\\
A_7=F_1\cap \bar{F_2},A_8=S_1\cap \bar{F_2},A_9=\bar F_1\cap \bar{F_2}.
\end{align*}
Then $w\in A_4$. When $S\neq \emptyset$, since $H'[S]$ is complete and $F_1\cap S=\emptyset$, without loss of generality, we may assume that $S\subseteq A_5\cup A_6\cup A_8\cup A_9$.
\begin{claim}\label{A1A7}
If $A_1\neq \emptyset$, then $|A_2\cup A_4\cup A_5|\geq k+1$ and $A_9=\emptyset$. Similarly, if $A_7\neq\emptyset$, then $|A_4\cup A_5\cup A_8|\geq k+1$ and $A_3=\emptyset$.
\end{claim}
\begin{proof}
Suppose that $A_1\neq \emptyset$. Since $\bar{F_1}\neq \emptyset$, $A_2\cup A_4\cup A_5$ is a cut set of $H'$. If $|A_2\cup A_4\cup A_5|<k+1$, since $\kappa(H')=k$, $|A_2\cup A_4\cup A_5|=k$. Now $A_1$ is a fragment of $H'$ to $A_2\cup A_4\cup A_5$ with $A_1\cap S=\emptyset$ and $w\in A_4\subseteq A_2\cup A_4\cup A_5$. Thus, $(A_2\cup A_4\cup A_5,A_1)$ is a good pair with $|A_1|<|F_1|$, contradicting the choice of $(S_1,F_1)$. So $|A_2\cup A_4\cup A_5|\geq k+1$, and consequently $|A_5\cup A_6\cup A_8|=|S_1|+|S_2|-|A_2\cup A_4\cup A_5|<k$. Thus, $A_9=\emptyset$. Otherwise $A_5\cup A_6\cup A_8$ is a cut set of $H'$ of size at most $k-1$, a contradiction. The second assertion follows by interchanging $F_2$ and $\bar F_2$.
\end{proof}
We first claim that $A_7=\emptyset$. If not, by Claim \ref{A1A7}, $|A_4\cup A_5\cup A_8|\geq k+1$ and $A_3=\emptyset$. Since $|A_2\cup A_5\cup A_8|=|S_1|=k$, $|A_4|>|A_2|$. Recall that $F_2\cap S=\emptyset$. Thus, $(S_2,F_2)$ is a good pair with $|F_2|=|A_1\cup A_2\cup A_3|<|A_1\cup A_4\cup A_7|=|F_1|$, contradicting the choice of $(S_1,F_1)$.

We next show that $A_1=\emptyset$. If not, Claim \ref{A1A7} implies $|A_2\cup A_4\cup A_5|\geq k+1$ and $A_9=\emptyset$. If $S\neq \emptyset$, then $|A_5\cup A_6\cup A_8|\geq k$ since $S\subseteq A_5\cup A_6\cup A_8$. Thus, $2k=|S_1|+|S_2|=|A_2\cup A_4\cup A_5|+|A_5\cup A_6\cup A_8|>2k$, a contradiction. If $S=\emptyset$, then $F_2\cap S=\emptyset$ and thus $(S_2,\bar{F_2})$ is a good pair. By Claim \ref{Fsize}, $|A_8|=|\bar{F_2}|\geq k+1$. Moreover, since $A_8\subseteq S_1$, we have $k+1\leq|A_8|\leq |S_1|=k$, a contradiction.

So we have $A_1,A_7=\emptyset$. Thus, $|A_4|=|F_1|\geq k+1$ by Claim \ref{Fsize}. Moreover, since $A_4\subseteq S_2$, we have $k+1\leq|A_4|\leq |S_2|=k$, a contradiction. This completes the proof of Lemma \ref{m=2}.
\end{proof}
\begin{corollary}\label{trim=2}
Let $k$ be a positive integer, let $G$ be a triangle-free graph with $\delta(G)\geq k+1$. The following statements hold.
\begin{enumerate}[label=(\arabic*)]
\item If $\kappa(G)=k$ and $F$ is an end of $G$ to $S_F$, then for every vertex $u\in F$ there exists $w\in N_G(u)\cap F$ such that $\kappa(G-\{u,w\})\geq k$.
\item If $\kappa(G)\geq k+1$, then for every vertex $u\in V(G)$ there exists $w\in N_G(u)$ such that $\kappa(G-\{u,w\})\geq k$.
\end{enumerate}
\end{corollary}
\begin{proof}
For $(1)$, since $|F|\geq \delta(G[F])+1\geq \delta(G)-|S_F|+1\geq 2$, by Lemma \ref{m=2} $(1)$, for every $u\in F$ there exists $w\in N_G(u)\cap F$ such that $\kappa(G\left \langle S_F \right \rangle-(\bar{F}\cup \{u,w\}))\geq k$. Since $\delta(G)\geq k+1=k+\frac{2}{2}$, applying Lemma \ref{com} to $uw$, we obtain that $\kappa(G-\{u,w\})\geq k$ and we are done. $(2)$ follows directly from Lemma \ref{m=2} $(2)$.
\end{proof}
\begin{lemma}\label{single}
Let $k$ be a positive integer, let $G$ be a triangle-free graph with $\kappa(G)=k$ and $\delta(G)\geq k+1$, and let $F$ be an end of $G$ to $S$. For every $u\in F$, $\kappa(G-\{u\})=k$ and there exists an end $F'\subseteq F$ of $G-\{u\}$ such that $N_G(u)\cap F'\neq \emptyset$.
\end{lemma}
\begin{proof}
Since $|F|\geq \delta(G[F])+1\geq (\delta(G)-|S|)+1\geq 2$, $\kappa(G\left \langle S \right \rangle-\bar{F})\geq k+1$ by Lemma \ref{mader[S]} and thus $\kappa(G\left \langle S \right \rangle-(\bar{F}\cup \{u\}))\geq k$. By Lemma \ref{com}, we obtain $\kappa(G-\{u\})\geq k$. Moreover, since $F\setminus \{u\}\neq \emptyset$, by Lemma \ref{subend}, $\kappa(G-\{u\})=k$ and there exists an end $F'\subseteq F$ of $G-\{u\}$ such that $N_G(u)\cap F'\neq \emptyset$. This completes the proof of Lemma \ref{single}.
\end{proof}
\begin{proof}[Proof of Theorem \ref{tripath}]
We prove the assertion for all positive integers $k$ by induction on $m$. For $m=1$, the assertion follows from Lemma \ref{single} when $\kappa(G)=k$, and $\kappa(G-\{u\})\geq k$ for any $u\in V(G)$ when $\kappa(G)\geq k+1$. For $m=2$, the assertion follows directly from Corollary \ref{trim=2}. Let $m\geq 3$ and assume that the assertion holds for all positive integers $k$ and all smaller values of $m$.

We first consider the case $\kappa(G)=k$. Let $F$ be an arbitrary end of $G$ and $u$ be an arbitrary vertex in $F$. Since $\delta(G)\geq k+\left \lceil \frac{m}{2} \right \rceil\geq k+2$, applying Lemma \ref{single} to $u$, we have that $\kappa(G-\{u\})=k$ and there exists an end $F_1\subseteq F$ of $G-\{u\}$ such that $N_G(u)\cap F_1\neq \emptyset$. Choose $v\in N_G(u)\cap F_1$. Since $\delta(G-\{u\})\geq \delta(G)-1\geq k+1$, applying Lemma \ref{single} again to $v$, we have $\kappa(G-\{u,v\})=k$ and there exists an end $F_2\subseteq F_1$ of $G-\{u,v\}$ such that $N_G(v)\cap F_2\neq \emptyset$. Choose $w\in N_G(v)\cap F_2$. As $uv\in E(G)$ and $G$ is triangle-free, no vertex is adjacent to both $u,v$. Thus, $\delta(G-\{u,v\})\geq k+\left \lceil \frac{m}{2} \right \rceil-1= k+\left \lceil \frac{m-2}{2} \right \rceil$. By the induction hypothesis, $F_2$ contains a path $P_w$ of order $m-2$ starting at $w$ such that $\kappa(G-(\{u,v\}\cup V(P_w)))\geq k$. The path $uvP_w$ is the desired path.

We next consider the case $\kappa(G)\geq k+1$. Fix an arbitrary vertex $v_1\in V(G)$, we successively extend a path $R_i=v_1\cdots v_i$ starting at $v_1$, stopping as soon as $i=m-2$ or $\kappa(G_i)=k$, where $G_0=G$ and $G_j=G-V(R_j)$ for every $j\in [i]$. This construction is possible. For every $j\in [m-3]$, since $G$ is triangle-free, $|N_G(v_j)\setminus V(R_j)|\geq \delta(G)-\left \lceil \frac{j-1}{2} \right \rceil>0$. Thus, whenever $\kappa(G_j)\geq k+1$, we can choose a neighbor $v_{j+1}$ of $v_j$ in $G-V(R_j)$ and then $\kappa(G_{j+1})\geq k$. By construction, $\kappa(G_j)\geq k+1$ for every $j<i$.

If $i=m-2$, then $R_i=v_1\cdots v_{m-2}$. Since $G$ is triangle-free, $|N_G(v_i)\cap V(G_i)|\geq \delta(G)-\left \lceil \frac{m-2}{2} \right \rceil\geq k+1$ and $\delta(G_i)\geq k+1$. If $\kappa(G_i)=k$, let $F_i$ be an end of $G_i$. Since $\kappa(G_{i-1})\geq k+1$, $N_G(v_i)\cap F_i\neq \emptyset$. Otherwise $N_{G_i}(F_i)$ is a $k$-cutset of $G_{i-1}$, contradicting $\kappa(G_{i-1})\geq k+1$. If $\kappa(G_i)\geq k+1$, let $F_i=V(G_i)$. In either case, choose $v_{m-1}\in N_G(v_i)\cap F_i$. By Corollary \ref{trim=2}, there exists $v_m\in N_{G_i}(v_{m-1})\cap F_i$ such that $\kappa(G_i-\{v_{m-1},v_m\})\geq k$. Hence, $v_1v_2\cdots v_m$ is the desired path.

We may assume that $i<m-2$. By the stopping rule, $\kappa(G_i)=k$. Since $\kappa(G_{i-1})\geq k+1$, $N_G(v_i)\cap F\neq \emptyset$, where $F$ is an end of $G_i$. Since $G$ is triangle-free, each vertex of $G$ has at most $\left \lceil \frac{i}{2} \right \rceil$ neighbors on $R_i$. Thus, $v_i$ has at least $k+\left \lfloor \frac{m-i}{2} \right \rfloor\geq k+1$ neighbors in $G_i$ and $\delta(G_i)\geq k+\left \lfloor \frac{m-i}{2} \right \rfloor$. Choose $v_{i+1}\in N_G(v_{i})\cap F$. If $m-i$ is even, the induction hypothesis implies that there exists a path $P\subseteq G[F]$ starting at $v_{i+1}$ of order $m-i$ such that $\kappa(G_{i}-V(P))\geq k$. Thus, the path $R_iP$ is the desired path. If $m-i\geq 3$ is odd, by Lemma \ref{single}, $G_{i+1}=G_i-\{v_{i+1}\}$ has connectivity $k$ and there exists an end $F'$ of $G_{i+1}$ with $N_G(v_{i+1})\cap F'\neq \emptyset$. Choose $v_{i+2}\in N_G(v_{i+1})\cap F'$. Since $\delta(G_{i+1})\geq k+\left \lfloor \frac{m-(i+1)}{2} \right \rfloor$ and $m-(i+1)$ is even, the induction hypothesis implies that there exists a path $P'$ starting at $v_{i+2}$ of order $m-(i+1)$ such that $\kappa(G_{i+1}-V(P'))\geq k$. Thus, the path $R_iv_{i+1}P'$ is the desired path.
\end{proof}
\section{Proof of Theorem \ref{triangle}}\label{trifulll}
We first consider $P_6^+(3)$, the tree whose three legs have orders $2,2,1$.
\begin{proposition}\label{P5+3}
Let $k$ be a positive integer, let $G$ be a triangle-free graph with $\kappa(G)\geq k$ and $\delta(G)\geq k+3$. Then $G$ contains a subgraph $T\cong P_6^+(3)$ such that $\kappa(G-V(T))\geq k$. Moreover, if $\kappa(G)=k$, then $T$ can be chosen such that $T\subseteq G[F]$, where $F$ is an arbitrary end of $G$.
\end{proposition}
\begin{proof}
If $\kappa(G)\geq k+1$, since $\delta(G)\geq k+3=(k+1)+\frac{4}{2}$, by Theorem \ref{tripath}, $G$ contains a path $v_1v_2v_3v_4$ such that $I=G-\{v_1,v_2,v_3,v_4\}$ is $(k+1)$-connected. Since $G$ is triangle-free and $v_1v_2v_3v_4$ is a path, $v_2$ has at least $\delta(G)-2\geq k+1$ neighbors in $I$. Choose $v_5\in N_G(v_2)\cap V(I)$. By Lemma \ref{m=2} $(2)$, there exists $v_6\in N_I(v_5)$ such that $I-\{v_5,v_6\}=G-\{v_1,v_2,v_3,v_4,v_5,v_6\}$ is $k$-connected. The tree with center $v_2$ and three legs $v_1$, $v_3v_4$, $v_5v_6$ is the desired tree.

If $\kappa(G)=k$, fix an end $F$ of $G$ and suppose that $G[F]$ contains no subgraph $T\cong P_6^+(3)$ such that $\kappa(G-V(T))\geq k$. Consider distinct vertices $u,v,x_1,x_2,x_3,x_4\in F$ such that $uv$ is an edge and $x_1x_2x_3x_4$ is a path. Let $U=\{u,v,x_1,x_2,x_3,x_4\}$. We call $(uv,x_1x_2x_3x_4,F')$ a {\it triple} if $\kappa(G-U)=k$ and $F'\subseteq F\setminus U$ is a nontrivial (i.e., $|F'|\geq 2$) end of $G-U$. It is {\it good} if $|N_G(F')\cap \{x_1,x_4\}|=1$.
\begin{claim}\label{uvdegree}
For every triple $(uv,x_1x_2x_3x_4,F')$, each of $u$ and $v$ has at most one neighbor in $\{x_1,x_2,x_3,x_4\}$.
\end{claim}
\begin{proof}
If $ux_2\in E(G)$, then the tree with center $x_2$ and three legs $x_1,uv,x_3x_4$ is the desired tree, a contradiction. By symmetry of $u$ and $v$,
$x_2$ and $x_3$, we obtain that $E(G[\{u,v\},\{x_2,x_3\}])=\emptyset$. If $\{x_1,x_4\}\subseteq N_G(u)$, then the tree with center $u$ and three legs
$v,x_1x_2,x_3x_4$ is the desired tree, a contradiction. By symmetry, $|N_G(v)\cap\{x_1,x_4\}|\leq 1$. Thus, each of $u$ and $v$ has at most one neighbor on $x_1x_2x_3x_4$.
\end{proof}
\begin{claim}\label{x2x3emp}
For every triple $(uv,x_1x_2x_3x_4,F')$, $N_G(F')\cap \{x_2,x_3\}=\emptyset$.
\end{claim}
\begin{proof}
Let $H=G-U$, let $S=N_H(F')$ and let $\bar{F'}=H-(S\cup F')$. By symmetry of $x_2$ and $x_3$, it suffices to show that $x_2\notin N_G(F')$. Suppose $x_2\in N_G(F')$ and choose $a\in N_G(x_2)\cap F'$. Since $|F'|\geq 2$, by Lemma \ref{m=2} $(1)$, there exists $b\in N_H(a)\cap F'$ such that $\kappa(H\left \langle S \right \rangle -(\{a,b\}\cup \bar{F'}))\geq k$.

Let $H_0=H-\{a,b\}$, $F_0=F'-\{a,b\}$ and $\bar{F_0}=\bar{F'}$.
Let $Q=G[V(H_0)\cup \{u,v\}]$. We assert that $\kappa(Q)<k$. Otherwise, the tree with center $x_2$ and three legs $x_1,x_3x_4,ab$ is the desired tree, a contradiction. Let $S'$ be a cut set of $Q$ with $|S'|<k$ and let $S''=S'\setminus \{u,v\}$.


We next verify the conditions of Lemma \ref{common} for $H_0$ and $S$. By Lemma \ref{mader[S]}, we obtain that $\kappa(H\left \langle S \right \rangle -\bar{F'})\geq k+1$ and $\kappa(H\left \langle S \right \rangle -F')\geq k$. Completing $S$ does not change the degrees of vertices in $F'$, so $d_H(a),d_H(b)\geq k+1$. As $N_H(a)\cap N_H(b)=\emptyset$ and $N_H(a)\cup N_H(b)\subseteq F'\cup S$, we obtain $|F'|\geq d_H(a)+d_H(b)-|S|\geq k+2$. Thus, $|F_0|=|F'|-2\geq k$. Since $F'\subseteq F$ and there are no edges between $F$ and $\bar{F}$, we have $\bar{F}\subseteq \bar{F_0}$. Thus, Lemma \ref{Fk+m} implies $|\bar{F_0}|\geq |\bar{F}|\geq k+6$. Moreover, $\kappa(H_0\left \langle S \right \rangle-\bar{F_0})=\kappa(H\left \langle S \right \rangle -(\{a,b\}\cup \bar{F'}))\geq k$ and $\kappa(H_0\left \langle S \right \rangle-F_0)=\kappa(H\left \langle S \right \rangle -F')\geq k$. Therefore, Lemma \ref{common} yields a component $C$ of $H_0-S''$ containing $(F_0\cup \bar{F_0})\setminus S''$.

Let $D$ be a component of $Q-S'$ not containing $C$.
Then $V(D)\subseteq S\cup\{u,v\}$. Since $x_1x_2x_3x_4$ is a path and $ab$ is an edge, $\delta(Q)\geq\delta(G)-3\geq k$.
Furthermore, Claim \ref{uvdegree} implies that each of $u$ and $v$ has at most one neighbor on $x_1x_2x_3x_4$. Thus $d_Q(u),d_Q(v)\geq\delta(G)-2\geq k+1$.

Let $r=|V(D)\cap\{u,v\}|$. Since $\delta(Q[V(D)])\geq\delta(Q)-|S'|\geq1$, we can choose an edge $yz\in E(Q[V(D)])$ such that $V(D)\cap\{u,v\}\subseteq \{y,z\}$. As $Q$ is triangle-free, $N_Q(y)$ and $N_Q(z)$ are disjoint, and both are contained in $V(D)\cup S'$.
Therefore,
\[
2k+r\leq d_Q(y)+d_Q(z)\leq |V(D)|+|S'|\leq k+r+(k-1)=2k+r-1,
\]
a contradiction. This completes the proof of Claim \ref{x2x3emp}.
\end{proof}
\begin{claim}\label{good}
$G$ contains a good triple.
\end{claim}
\begin{proof}
Since $\delta(G)\geq k+3\geq k+\frac{6}{2}$, there exists a path $P=y_1y_2y_3y_4\subseteq G[F]$ such that $\kappa(G-V(P))\geq k$ by Theorem \ref{tripath} $(1)$, and $|F|\geq k+6$ by Lemma \ref{Fk+m} and thus $F\setminus V(P)\neq \emptyset$. By Lemma \ref{subend}, we have that $H_1=G-V(P)$ has connectivity $k$ and there exists an end $F_1\subseteq F$ such that $N_G(F_1)\cap V(P)\neq \emptyset$. Since $G$ is triangle-free and $P$ is a path, $\delta(H_1)\geq \delta(G)-2\geq k+1$.

For every $z\in F_1$, we assert that $y_2,y_3\notin N_G(z)$ and $z$ has at most one neighbor in $\{y_1,y_4\}$. Since $\delta(H_1)\geq k+1$, by Corollary \ref{trim=2} $(1)$, there exists $z'\in N_{H_1}(z)\cap F_1$ such that $H_1-\{z,z'\}$ is $k$-connected. Suppose $y_2\in N_G(z)$. Then the tree with center $y_2$ and three legs $y_1,y_3y_4,zz'$ is the desired tree, a contradiction. By symmetry of $y_2$ and $y_3$, we also have $y_3\notin N_G(z)$. Suppose $\{y_1,y_4\}\subseteq N_G(z)$. Then the tree with center $z$ and three legs $z',y_1y_2,y_3y_4$ is the desired tree, a contradiction. So every vertex of $F_1$ has at most one neighbor on $P$, and $|F_1|\geq \delta(H_1[F_1])+1\geq \delta(G)-k\geq3$.

Reversing $P$ if necessary, choose $a\in N_G(y_4)\cap F_1$. By Corollary \ref{trim=2} $(1)$, there exists $b\in N_{H_1}(a)\cap F_1$ such that $\kappa(H_1-\{a,b\})\geq k$. Moreover, since $F_1\setminus \{a,b\}\neq \emptyset$, Lemma \ref{subend} implies that $H_2=H_1-\{a,b\}$ has connectivity $k$ and there exists an end $F_2\subseteq F_1$ of $H_2$ such that $N_G(F_2)\cap \{a,b\}\neq \emptyset$. As $G$ is triangle-free, every vertex of $F_2$ has at most one neighbor in $\{a,b\}$. So $|F_2|\geq \delta(H_2[F_2])+1\geq \delta(G)-2-k+1\geq 2$. Thus $(y_1y_2,y_3y_4ab,F_2)$ is a triple.

It follows from Claim \ref{x2x3emp} that $a\notin N_G(F_2)$. Thus, $b\in N_G(F_2)$. Since $y_3\notin N_G(F_1)$ and $F_2\subseteq F_1$, $y_3\notin N_G(F_2)$. Hence, $(y_1y_2,y_3y_4ab,F_2)$ is a good triple.
\end{proof}

Choose a good triple $(ab,cdef,F^*)$ with $|F^*|$ as small as possible. Reversing $cdef$ if necessary, assume $f\in N_G(F^*)$ and $c\notin N_G(F^*)$.
By Claim \ref{x2x3emp}, $N_G(F^*)\cap \{d,e\}=\emptyset$.

Let $H=G-\{a,b,c,d,e,f\}$ and let $Q=G-\{a,b,e,f\}$. Since $G$ is triangle-free and $ab$, $cdef$ are both paths, every vertex of $G$ has at most three neighbors in $\{a,b,c,d,e,f\}$ and thus each of $c$ and $d$ has at least $\delta(G)-3\geq k$ neighbors in $H$. Recall that $N_G(F^*)\cap \{c,d\}=\emptyset$. By Lemma \ref{lem:add-vertex-outside-end}, $\kappa(Q)=k$ and $F^*$ remains an end of $Q$. Since $ab,ef$ are both edges, $\delta(Q)\geq \delta(G)-2\geq k+1$.

We assert that every $z\in F^*$ has at most one neighbor in $\{a,b,f\}$. Suppose, to the contrary, that some $z\in F^*$ has at least two neighbors in $\{a,b,f\}$.  Since $ab\in E(G)$ and $G$ is triangle-free, $z$ cannot be adjacent to both $a$ and $b$. By symmetry, we may assume that $za,zf\in E(G)$. Since $\delta(Q)\geq k+1$, by Corollary \ref{trim=2} $(1)$, there exists $z'\in N_Q(z)\cap F^*$ such that $G-\{a,b,e,f,z,z'\}=Q-\{z,z'\}$ is $k$-connected. Then the tree with center $z$ and three legs $z',ab,ef$ is the desired tree, a contradiction. Consequently, every vertex of $F^*$ has at most one neighbor in $\{a,b,e,f\}$ and $|F^*|\geq \delta(G)-1-k+1\geq 3$.

Choose $g\in N_G(f)\cap F^*$. Corollary \ref{trim=2} $(1)$ implies that there exists $h\in N_Q(g)\cap F^*$ such that $O=Q-\{g,h\}$ is $k$-connected. As $F^*\setminus \{g,h\}\neq \emptyset$, by Lemma \ref{subend}, $\kappa(O)=k$ and there exists an end $F'\subseteq F^*$ of $O$ such that $N_G(F')\cap \{g,h\}\neq \emptyset$.

It follows from $F'\subseteq F^*$ that every vertex of $F'$ has at most one neighbor in $\{a,b,e,f\}$. Moreover, since $G$ is triangle-free and $gh$ is a path, every vertex of $F'$ has at most one neighbor in $\{g,h\}$. Thus, $|F'|\geq \delta(O[F'])+1\geq \delta(G)-2-k+1\geq 2$. Hence, $(ab,efgh,F')$ is a triple. By Claim \ref{x2x3emp}, $g\notin N_G(F')$ and thus $h\in N_G(F')$. We also have $e\notin N_G(F')$ since $F'\subseteq F^*$. Then $|N_G(F')\cap \{e,h\}|=1$. Hence, $(ab,efgh,F')$ is a good triple with $|F'|<|F^*|$, a contradiction.
\end{proof}
Let $P_s^+(i)$ be a non-path tree with center $c$ and three legs $L_j=v_j^1\cdots v_j^{l_j}$ for $j\in[3]$, where $l_1\geq l_2\geq l_3=1$ and $l_1\geq 3$. Write $v_j^0=c$. Then the following proposition holds.
\begin{proposition}\label{s}
Let $k$ be a positive integer, let $G$ be a triangle-free graph with $\kappa(G)=k$ and $\delta(G)\geq k+\max(P_s^+(i))$ and let $F$ be an arbitrary end of $G$. If $G[F]$ contains a subgraph $T\cong P_s^+(i)-\{v_1^{l_1},v_1^{l_1-1}\}$ such that $\kappa(G-V(T))\geq k$, then $G[F]$ contains a subgraph $H\cong P_s^+(i)$ such that $\kappa(G-V(H))\geq k$.
\end{proposition}
\begin{proof}
Suppose, to the contrary, that $G$ is a triangle-free graph with $\kappa(G)=k$ and $\delta(G)\geq k+\max(P_s^+(i))$, $G$ contains an end $F^\star$ such that $G[F^\star]$ contains a subgraph $T\cong P_s^+(i)-\{v_1^{l_1},v_1^{l_1-1}\}$ with $\kappa(G-V(T))\geq k$ and $G[F^\star]$ contains no subgraph $H\cong P_s^+(i)$ with $\kappa(G-V(H))\geq k$.

Let $\phi$ be an isomorphism from $P_s^+(i)-\{v_1^{l_1},v_1^{l_1-1}\}$ to $T$. Write $\phi(c)=c'$ and $\phi(v_j^r)=u_j^r$ for every vertex $v_j^r\in V(P_s^+(i)-\{v_1^{l_1-1},v_1^{l_1}\})$. Write $L_1'=u_1^1\cdots u_1^{l_1-2}$, $L_2'=u_2^1\cdots u_2^{l_2}$, and $L_3'=u_3^1$. Also, write $u_1^0=c'$.

\begin{claim}\label{degreek+1}
Let $T^\star$ be a tree in $G$ isomorphic to $P_s^+(i)-V(L_1[v_1^j,v_1^{l_1}])$, where $j\in[l_1-1]$. Then $|N_G(u)\cap V(G-V(T^\star))|\ge k+\left \lfloor \frac{l_1-j+1}{2} \right \rfloor$ for every vertex $u\in V(G)$.
\end{claim}
\begin{proof}
Choose an arbitrary vertex $u\in V(G)$. By Lemma \ref{vT^*}, $|N_G(u)\cap V(T^\star)|\leq \max(T^\star)$. Therefore,
\begin{align*}
|N_G(u)\cap V(G-V(T^\star))|&\geq \delta(G)-\max(T^\star)\geq k+\max(P_s^+(i))-\max(T^\star)\\
&=k+(\left \lceil \frac{l_1}{2}\right \rceil+\left \lceil \frac{l_2}{2}\right \rceil+1)-(\left \lceil \frac{j-1}{2}  \right \rceil+\left \lceil \frac{l_2}{2}\right \rceil+1)\\
&\geq k+\left \lfloor \frac{l_1-j+1}{2} \right \rfloor.
\end{align*}
\end{proof}
\begin{claim}\label{degreel_1}
For any $j\in [l_1-1]$, if $G[F^\star]$ contains a subgraph $T^\star\cong P_s^+(i)-V(L_1[v_1^j, v_1^{l_1}])$ with an isomorphism $\sigma: P_s^+(i)-V(L_1[v_1^j,v_1^{l_1}])\to T^\star$ such that $\kappa(G-V(T^\star))=k$ and $N_G(\sigma(v_1^{j-1}))\cap V(F_j)\neq\emptyset$, where $F_j\subseteq F^\star$ is an end of $G-V(T^\star)$, then $G[F^\star]$ contains a subgraph $H \cong P_s^+(i)$ such that $\kappa(G-V(H))\ge k$.
\end{claim}
\begin{proof}
Choose $w\in N_G(\sigma(v_1^{j-1}))\cap F_j$. By Claim \ref{degreek+1}, $\delta(G-V(T^\star))\geq k+\left \lfloor \frac{l_1-j+1}{2} \right \rfloor$. If $l_1-j+1$ is even, Theorem \ref{tripath} implies that $G[F_j]$ contains a path $P_w$ of order $l_1-j+1$ starting from $w$ such that $\kappa(G-(V(T^\star)\cup V(P_w)))\geq k$. Joining $P_w$ to $T^\star$ by the edge $\sigma(v_1^{j-1})w$ gives $H$.

If $l_1-j+1$ is odd, we have $l_1-j+1\geq 3$ since $j\leq l_1-1$. Thus, $\delta(G-V(T^\star))\geq k+\left \lfloor \frac{l_1-j+1}{2} \right \rfloor\geq k+1$. By Lemma \ref{single}, $\kappa(G-(V(T^\star)\cup \{w\}))=k$ and $G-(V(T^\star)\cup \{w\})$ contains an end $F'\subseteq F_j$ such that $N_G(w)\cap V(F')\neq \emptyset$. Let $T'$ be obtained from $T^\star$ by adding $w$ and the edge $\sigma(v_1^{j-1})w$. Then $T'\subseteq G[F^\star]$ is isomorphic to $P_s^+(i)-V(L_1[v_1^{j+1},v_1^{l_1}])$ with $v_1^j$ corresponding to $w$. Applying Claim \ref{degreek+1} to $T'$ with the index $j+1$, we obtain $\delta(G-V(T'))\ge k+\left\lfloor\frac{l_1-j}{2}\right\rfloor$. Since $l_1-j$ is an even integer at least $2$, the preceding even case applies.
\end{proof}
Since $\delta(G)\geq k+\max(P_s^+(i))\geq k+\lceil\frac{s}{2}\rceil$, by Lemma \ref{Fk+m}, $|F^\star|\geq k+s$. Let $G_1=G-V(T)$. Since $F^\star\setminus V(T)\neq \emptyset$ and $\kappa(G_1)\geq k$, by Lemma \ref{subend}, $\kappa(G_1)=k$ and there exists an end $F\subseteq F^\star$ of $G_1$ such that $N_G(T)\cap F\neq\emptyset$.
\begin{claim}\label{Xneighbor}
For every $X\subseteq V(T)$ with $N_G(X)\cap F=\emptyset$, the graph $G[V(G_1)\cup X]$ has connectivity $k$ and still has $F$ as an end.
\end{claim}
\begin{proof}
It follows from Claim \ref{degreek+1} with $j=l_1-1$ that every vertex of $T$ has at least $k+1$ neighbors in $G_1$. Consequently, for every $X\subseteq V(T)$ with $N_G(X)\cap F=\emptyset$, Lemma \ref{lem:add-vertex-outside-end} implies that $G[V(G_1)\cup X]$ has connectivity $k$ and still has $F$ as an end.
\end{proof}

We distinguish the following three cases according to the adjacency between $T$ and $F$:

{\bf Case A.} $N_G(c'L_1')\cap F\neq \emptyset$.

Let $h$ be the largest index such that $N_G(u_1^h)\cap F\neq\emptyset$, where $0\leq h\leq l_1-2$. Let $X=\{u_1^r:h+1\leq r\leq l_1-2\}$ and $T'=T-X$. By the choice of $h$, $N_G(X)\cap F=\emptyset$. Thus $\kappa(G-V(T'))=k$ and $F$ remains an end of $G-V(T')$ by Claim \ref{Xneighbor}. Applying Claim \ref{degreel_1} with $j=h+1$ gives a contradiction.

{\bf Case B.} $N_G(c'L_1')\cap F=\emptyset$ and $N_G(L_3')\cap F\neq \emptyset$.

Let $T'=T-\{u_1^r:2\leq r\leq l_1-2\}$. Then $\kappa(G-V(T'))=k$ and $F$ remains an end of $G-V(T')$ by Claim \ref{Xneighbor}. After interchanging the roles of $u_1^1$ and $u_3^1$ in $T'$, we obtain a contradiction from Claim \ref{degreel_1} with $j=2$.

{\bf Case C.} $N_G(c'L_1'L_3')\cap F=\emptyset$ and $N_G(V(L_2'))\cap F\neq \emptyset$.

Let $h$ be the largest index such that $N_G(u_2^h)\cap F\neq\emptyset$. Let $Q$ be the subpath of $L_2'[u_2^h,u_2^1]c'L_1'$ of order $\min\{h+l_1-1,l_2+2\}$ starting at $u_2^h$. Since $h\leq l_2$ and $l_1\geq3$, the other endvertex of $Q$ lies in $L_1'$. By the choice of $h$, exactly one endvertex of $Q$ has a neighbor in $F$ and $N_G(T-V(Q))\cap F=\emptyset$. By Claim \ref{Xneighbor}, $\kappa(G-V(Q))=k$ and $F$ remains an end of $G-V(Q)$.

If $|Q|<l_2+2$, we extend $Q$ as follows. Since $G$ is triangle-free, for every path $R$ of order at most $l_2+2$, $\delta(G-V(R))
\geq (k+\left \lceil \frac{l_1}{2} \right \rceil +\left \lceil \frac{l_2}{2} \right \rceil +1)-\left \lceil \frac{l_2+2}{2} \right \rceil\geq k+1$.
Let $a$ be the endvertex of $Q$ having a neighbor $x$ in $F$. By Lemma \ref{single}, $\kappa(G-(V(Q)\cup\{x\}))=k$ and $G-(V(Q)\cup\{x\})$ contains an end
$F'\subseteq F$ such that $N_G(x)\cap V(F')\neq\emptyset$. Extending $Q$ by the edge $ax$, we obtain a longer path with exactly one endvertex having a neighbor in $F'$. Replacing $Q$ by this longer path and $F$ by $F'$, we can repeat this step until the path has order $l_2+2$.

We may choose a path $P=x_1\cdots x_{l_2+2}\subseteq G[F^\star]$ and an end $F_P\subseteq F^\star$ of $G-V(P)$  with $|F_P|$ as small as possible such that $\kappa(G-V(P))=k$ and exactly one endvertex of $P$ has a neighbor in $F_P$. Without loss of generality, assume that $N_G(x_1)\cap V(F_P)=\emptyset$ and $N_G(x_{l_2+2})\cap V(F_P)\neq\emptyset$.

We assert that $N_G(x_2)\cap F_P=\emptyset$. Otherwise, view $P$ as a copy of $P_s^+(i)-V(L_1)$ with $c$ corresponding to $x_2$. Applying Claim \ref{degreel_1} with $j=1$ then gives a contradiction.

Choose $y\in N_G(x_{l_2+2})\cap F_P$. Since $\delta(G-V(P))\geq k+\left \lceil \frac{l_1}{2} \right \rceil\geq k+1$, by Lemma \ref{single}, $\kappa(G-(V(P)\cup \{y\}))=k$ and $G-(V(P)\cup \{y\})$ contains an end $F'\subseteq F_P$ such that $N_G(y)\cap V(F')\neq\emptyset$.
Let $P'=x_2\cdots x_{l_2+2}y$.  The vertex $x_1$ has no neighbor in $F'$ and has at least $\delta(G)-\left \lceil \frac{l_2+2}{2} \right \rceil-1\geq k+\left \lceil \frac{l_1}{2} \right \rceil-1\geq k$ neighbors in $G-(V(P)\cup \{y\})$. By Lemma \ref{lem:add-vertex-outside-end}, adjoining $x_1$ implies $\kappa(G-V(P'))=k$ and $F'$ is still an end of $G-V(P')$. Since $N_G(x_2)\cap V(F_P)=\emptyset$, exactly one endvertex of $P'$ has a neighbor in $F'$. Together with $P'\subseteq F^\star$ and $F'\subseteq F^\star$, this contradicts the choice of $P,F_P$ as $|F'|<|F_P|$.

This completes the proof of Proposition \ref{s}.
\end{proof}
\begin{proof}[Proof of Theorem \ref{triangle}]
We prove the assertion for all positive integers $k$ by induction on the order of $P_m^+(i)$. When $P_m^+(i)$ is a path, the assertion holds by Theorem \ref{tripath}.

Proposition \ref{P5+3} settles the case $P_6^+(3)$, and also resolves the two cases $P_4^+(2)$ and $P_5^+(2)$. To see this, the minimum degree conditions are the same since $\max(P_6^+(3))=\max(P_4^+(2))=\max(P_5^+(2))=3$. Choose a subgraph $R\cong P_6^+(3)$ such that $\kappa(G-V(R))\geq k$. When $\kappa(G)=k$, choose $R$ inside the prescribed end. Choose a subgraph $T\subseteq R$ isomorphic to $P_4^+(2)$ or $P_5^+(2)$. By Lemma \ref{vT^*}, every vertex of $R\setminus V(T)$ has at least $\delta(G)-3\geq k$ neighbors in $G-V(R)$. Adding these vertices back implies $\kappa(G-V(T))\geq k$.

All remaining non-path cases have legs $L_j=v_j^1\cdots v_j^{l_j}$ for $j\in[3]$ with $l_1\geq l_2\geq l_3=1$ and $l_1\geq 3$. Note that $\max(P_m^+(i)-\{v_1^{l_1},v_1^{l_1-1}\})=\max(P_m^+(i))-1$.

If $\kappa(G)=k$, since $\delta(G)>k+\max(P_m^+(i)-\{v_1^{l_1},v_1^{l_1-1}\})$, we apply the induction hypothesis together with Proposition \ref{s} to obtain the desired conclusion.

If $\kappa(G)\geq k+1$, since $\delta(G)\geq (k+1)+\max(P_m^+(i)-\{v_1^{l_1},v_1^{l_1-1}\})$, the induction hypothesis implies that $G$ contains a subgraph $H\cong P_m^+(i)-\{v_1^{l_1},v_1^{l_1-1}\}$ with an isomorphism $\ell:P_m^+(i)-\{v_1^{l_1},v_1^{l_1-1}\}\to H$ such that $\kappa(G-V(H))\geq k+1$. By Lemma \ref{vT^*}, each vertex of $G$ has at most $\max(P_m^+(i)-\{v_1^{l_1},v_1^{l_1-1}\})=\max(P_m^+(i))-1$ neighbors in $H$. Thus, $\ell(v_1^{l_1-2})$ has at least $k+\max(P_m^+(i))-(\max(P_m^+(i))-1)=k+1$ neighbors in $G-V(H)$ and $\delta(G-V(H))\geq k+1$. Choose $b\in N_G(\ell(v_1^{l_1-2}))\cap V(G-V(H))$. By Lemma \ref{m=2} $(2)$, there exists $b'\in N_{G-V(H)}(b)$ such that $\kappa(G-(V(H)\cup \{b,b'\}))\geq k$. Joining $bb'$ to $H$ by the edge $\ell(v_1^{l_1-2})b$ gives the desired subgraph isomorphic to $P_m^+(i)$.
\end{proof}

\end{spacing}


\begin{thebibliography}{99}


\bibitem{Bondy} J.A. Bondy and U.S.R. Murty, Graph Theory, Springer, New York, 2007.

\bibitem{Chartrand} G. Chartrand, A. Kaugars, and D.R. Lick, Critically $n$-connected graphs, Proc. Am. Math. Soc. 32 (1972) 63-68.

\bibitem{CFPR} H. Chu, S. Fujita, B. Park and H. Ryu, Connectivity keeping trees in triangle-free graphs, arXiv:2511.06622 (2025).

\bibitem{Diwan} A.A. Diwan and N.P. Tholiya, Non-separating trees in connected graphs, Discrete Math. 309 (2009) 5235-5237.

\bibitem{Ftriangle} S. Fujita, Connectivity keeping paths containing prescribed vertices in highly connected triangle-free graphs, J. Combin. Theory Ser. B 174 (2025) 190-206.

\bibitem{Fujita}  S. Fujita and K. Kawarabayashi, Connectivity keeping edges in graphs with large minimum degree, J. Combin. Theory Ser. B 98 (2008) 805-811.

\bibitem{Hasunuma2}  T. Hasunuma and K. Ono, Connectivity keeping trees in $2$-connected graphs, J. Graph Theory 94 (2020) 20-29.

\bibitem{HongLiu} Y. Hong and Q. Liu, Mader's conjecture for graphs with small connectivity, J. Graph Theory 101 (3) (2022) 379-388.

\bibitem{HongLLY} Y. Hong, Q. Liu, C. Lu, and Q. Ye, Connectivity keeping caterpillars and spiders in $2$-connected graphs, Discrete Math. 344 (2021), 112236.

\bibitem{HongWuLiu} Y. Hong, Y. Wu, and Q. Liu, Redundant Trees in Bipartite Graphs, Mathematics 13 (6) (2025) 1005.

\bibitem{Ji} M. Ji, Connectivity keeping paths for $k$-connected bipartite graphs, Discrete Mathematics, Algorithms and Applications 17 (1) (2025) 2430003.
    
\bibitem{Ji2} M. Ji, Non-path results on the connectivity keeping problem, Discussiones Math. Graph Theory, 45(4) (2025) 1163-1184.

\bibitem{LYH} Q. Liu, K. Ying, and Y. Hong, Highly connected triples and Mader's conjecture, J. Graph Theory 107 (2024) 478-484.

\bibitem{Lu}  C. Lu and P. Zhang, Connectivity keeping trees in $2$-connected graphs, Discrete Math. 343 (2020), no. 2, 1-4.

\bibitem{Luo} L. Luo, Y. Tian, and L. Wu, Connectivity keeping paths in $k$-connected bipartite graphs, Discrete Math. 345 (2022), no. 4, 112788.

\bibitem{Mader1} W. Mader, Connectivity keeping paths in $k$-connected graphs, J. Graph Theory 65 (2010) 61-69.

\bibitem{Mader2} W. Mader, Connectivity keeping trees in $k$-connected graphs, J. Graph Theory 69 (2012) 324-329.

\bibitem{Tian3} Y. Tian, J. Meng, A survey on the vertex-(edge-)$k$-maximal graphs and the $k$-vertex-(edge-)connected graphs with redundant subgraphs, Discrete Appl. Math. 378 (2026), 125-135.

\bibitem{Tian1} Y. Tian, J. Meng, H. Lai, and L. Xu, Connectivity keeping stars or double-stars in $2$-connected graphs, Discrete Math. 341 (4) (2018) 1120-1124.

\bibitem{Tian2} Y. Tian, H. Lai, L. Xu, and J. Meng, Nonseparating trees in $2$-connected graphs and oriented trees in strongly connected digraphs, Discrete Math. 342 (2) (2019) 344-351.

\bibitem{YangTian} Q. Yang, Y. Tian, Proof of a conjecture on connectivity keeping odd paths in k-connected bipartite
graphs, Discrete Math. 348 (2025) 114476.
\end{thebibliography}
\end{document}